\documentclass[12pt,reqno]{amsart}

\usepackage{amsfonts,amsmath,amssymb,amsthm,mathtools}
\usepackage{mathrsfs}
\usepackage{enumerate}
\usepackage[hidelinks]{hyperref}
\usepackage{microtype}

\allowdisplaybreaks
\numberwithin{equation}{section}

\newcommand{\Nat}{\mathbb{N}}
\newcommand{\Intg}{\mathbb{Z}}

\newcommand{\Ccat}{\mathcal{C}}
\newcommand{\G}{\mathcal{G}}

\newcommand{\Tcal}{\mathcal{T}}

\newcommand{\id}{\mathrm{id}}

\newcommand{\ip}[2]{\langle #1,#2\rangle}

\newcommand{\bm}[1]{\boldsymbol{#1}}
\newcommand{\cstar}[1]{C^*(#1)}

\newcommand{\compacts}{\mathcal{K}}
\newcommand{\adjointables}{\mathcal{L}}
\newcommand{\MCE}{\operatorname{MCE}}
\newcommand{\NT}{\mathcal{NT}}
\newcommand{\NOcal}{\mathcal{NO}}

\theoremstyle{plain}
\newtheorem{theorem}{Theorem}[section]
\newtheorem{lemma}[theorem]{Lemma}
\newtheorem{corollary}[theorem]{Corollary}
\newtheorem{proposition}[theorem]{Proposition}

\theoremstyle{definition}
\newtheorem{definition}[theorem]{Definition}
\newtheorem{example}[theorem]{Example}
\newtheorem{remark}[theorem]{Remark}
\newtheorem{question}[theorem]{Question}

\title[Weak factorization over groupoids]
{Weak factorization and product systems over groupoids}
\author[Bannon]{Jon Bannon}
\address{Department of Mathematics, Siena University, 515 Loudon Road,
Loudonville, NY 12211, USA}
\email{jbannon@siena.edu}
\author[Vdovina]{Alina Vdovina}
\address{Department of Mathematics, City College of New York, CUNY,
160 Convent Avenue, New York, NY 10031, USA}
\email{avdovina@ccny.cuny.edu}

\hypersetup{
  pdftitle={Weak factorization and product systems over groupoids},
  pdfauthor={Jon Bannon and Alina Vdovina},
  pdfsubject={Generalized higher-rank graphs and product systems},
  pdfkeywords={generalized higher-rank graph, product system, groupoid biset,
    weak factorization, Zappa--Sz\'ep product}
}


\begin{document}

\begin{abstract}
Let $(\Ccat,\ell)$ be a generalized higher-rank graph and let
$\G=\ell^{-1}(\bm 0)$ be its groupoid of invertible arrows.  We show that the
weak factorization property is equivalent to the assertion that multiplication
induces coherent bijections
\[
  \Ccat^{\bm m}*_{\G}\Ccat^{\bm n}
       \longrightarrow \Ccat^{\bm m+\bm n}
\]
of balanced groupoid bisets.  Thus generalized higher-rank graphs with fixed
invertible groupoid are the same data as normalized $\Nat^k$-product systems
of $\G$--$\G$-bisets.  If $\Ccat$ is countable and left cancellative, these
bisets linearize canonically to a product system of
$C^*(\G)$-correspondences.  The multiplication unitaries are induced by
category multiplication.  Finite alignment of $\Ccat$ implies compact
alignment of the correspondence system, and its Nica--Toeplitz algebra is
canonically isomorphic to Spielberg's full category algebra
$C^*(\mathscr G_2(\Ccat))$.
Under row-finiteness modulo $\G$ and no sources, Fowler's Cuntz--Pimsner
quotient is the boundary reduction of this groupoid; with injective left
actions, the same holds for the Cuntz--Nica--Pimsner algebra.  We characterize the
$R$-condition as the existence of a strict multiplicative splitting by
right-orbit representatives and identify the resulting correspondence system
with the system attached to a higher-rank graph/groupoid Zappa--Sz\'ep
product.  A cancellative rank-two example with a nontrivial group-valued
commuting square admits no strict splitting.
\end{abstract}

\maketitle

\section{Introduction}
\label{sec:introduction}

The factorization property of a higher-rank graph replaces the unique
decomposition of a word at a prescribed length.  Lawson and Vdovina introduced
generalized higher-rank graphs by allowing the two pieces of a factorization to
be modified by an invertible arrow at their common vertex
\cite[Section~3]{LawsonVdovina2022}.  More precisely, if
$\ell(c)=\bm m+\bm n$, then $c=ab$ with $\ell(a)=\bm m$ and
$\ell(b)=\bm n$, and any other such factorization has the form
\[
             c=(ag)(g^{-1}b)
\]
for an invertible arrow $g$.  The connecting arrow need not be unique.

The quotient implicit in this statement is the balanced product over the
groupoid of invertible arrows.  Put
\[
       \G=\Ccat^\times=\ell^{-1}(\bm 0),
       \qquad \Ccat^{\bm n}=\ell^{-1}(\bm n).
\]
Each homogeneous fiber is a $\G$--$\G$-biset under multiplication.  Theorem
\ref{thm:wfp-biset-system} proves that the weak factorization property is
equivalent to bijectivity of
\begin{equation}
 \label{eq:intro-balanced-multiplication}
 \Ccat^{\bm m}*_{\G}\Ccat^{\bm n}
       \longrightarrow \Ccat^{\bm m+\bm n},
       \qquad [a,b]\longmapsto ab.
\end{equation}
Associativity in $\Ccat$ is precisely the coherence of these bijections.
Conversely, a normalized $\Nat^k$-product system of $\G$--$\G$-bisets
reconstructs the generalized higher-rank graph.  For fixed $\G$, the two
constructions give an isomorphism of categories.

The biset statement has a direct operator-algebraic linearization.  Suppose
that $\Ccat$ is countable and left cancellative, and set $A=\cstar{\G}$.  Left
cancellation makes the right action of $\G$ on every fiber free.  For
$c,d\in\Ccat^{\bm n}$, let $\xi_c$ denote the corresponding elementary vector
and set
\begin{equation}
 \label{eq:intro-inner-product}
 \ip{\xi_c}{\xi_d}
 =
 \begin{cases}
   u_g,&d=cg\text{ for }g\in\G,\\
   0,&d\notin c\G,
 \end{cases}
 \qquad
 u_h\xi_c=
 \begin{cases}
   \xi_{hc},&s(h)=r(c),\\
   0,&\text{otherwise}.
 \end{cases}
\end{equation}
Here $u_g$ is the canonical partial isometry in $A$.  The completion is an
$A$-correspondence $X_{\bm n}$.  Theorem~\ref{thm:linearization} shows that
category multiplication extends to unitaries
\begin{equation}
 \label{eq:intro-correspondence-multiplication}
 X_{\bm m}\otimes_A X_{\bm n}\longrightarrow X_{\bm m+\bm n},
 \qquad \xi_a\otimes\xi_b\longmapsto \xi_{ab},
\end{equation}
where a noncomposable elementary tensor is zero.  The inner-product calculation
uses weak factorization in the converse direction.  No multiplicative
transversal enters the construction.  If every
nonzero homogeneous fiber has finitely many right $\G$-orbits, the
correspondences are finitely generated projective and the resulting product
system is compactly aligned.

For a self-similar group action on a finite alphabet, $X_1$ is the
full-$C^*(G)$ completion of Nekrashevych's algebraic self-similarity bimodule,
with left action given by his linear recursion
\cite[pp.~227--228]{Nekrashevych2004}.  For several degrees, the extra data
are factorization unitaries between the degree-one correspondences.  When the generalized higher-rank
graph satisfies the $R$-condition, Lawson--Vdovina obtain a higher-rank graph
$\Lambda$, a size-preserving Zappa--Sz\'ep action, and a unique normal form
$c=\lambda g$ \cite[Theorems~4.2 and~4.8]{LawsonVdovina2022}.  We show in
Theorem~\ref{thm:r-condition-splitting} that this is exactly a strict
multiplicative splitting of the biset product system.  Under this splitting,
our correspondence system has fibers
\[
 X(\Lambda^{\bm n})\otimes_{C_0(\Lambda^0)}\cstar{\G}
\]
and the Zappa--Sz\'ep multiplication formula.  For a finite $\Lambda$ and a
faithful self-similar action, this is the compactly aligned product system of
Afsar, Brownlowe, Ramagge, and Whittaker
\cite[Proposition~5.1]{AfsarBrownloweRamaggeWhittaker2026}.  Faithfulness is
part of their definition of a self-similar groupoid action; it is not a
consequence of the $R$-condition.  Our construction does not require it.

Strict splittings need not exist, even for cancellative categories.  In
Section~\ref{sec:twisted-example} we give a generalized
rank-two graph on $\Nat^2\times\Intg$ whose degree-zero group is $\Intg$ and
whose multiplication contains the cocycle $-(\text{second coordinate})
(\text{first coordinate})$.  Every atomic transversal violates the
$R$-condition.  Its correspondence fibers nevertheless all have rank one over
$C^*(\Intg)\cong C(\mathbb T)$, and the two colors commute through the
canonical unitary of $C^*(\Intg)$.

The correspondence system also recovers the operator algebras of the
underlying category.  If $\Ccat$ is finitely aligned, Proposition
\ref{prop:finite-alignment-compacts} proves compact alignment by an explicit
minimal-common-extension formula for products of rank-one operators.  Theorem
\ref{thm:nica-spielberg} then identifies its Nica--Toeplitz algebra with
Spielberg's full category algebra
$C^*(\mathscr G_2(\Ccat))$ \cite[Definition~5.13 and
Theorem~9.7]{Spielberg2020}.  Under row-finiteness modulo $\G$ and no sources,
Theorem~\ref{thm:fowler-spielberg-boundary} identifies Fowler's
Cuntz--Pimsner quotient with the boundary reduction
$C^*(\mathscr G_2(\Ccat)|_{\partial\Ccat})$.  If the left actions are also
injective, Sims--Yeend covariance gives the same identification for the
Cuntz--Nica--Pimsner algebra
\cite[Proposition~5.1 and Corollary~5.2]{SimsYeend2010}.  The injectivity
hypothesis is needed for this application of their result.

Replacing $C^*(\G)$ by a coefficient quotient raises a compatibility question.
Nekrashevych's minimal self-similar completion is generally a quotient of the full group algebra
\cite[Theorem~5.2]{Nekrashevych2004}; a higher-rank analogue must ensure
compatibility in every degree and with every multiplication unitary.  We
formulate this issue in Section~\ref{sec:further-questions}.
 \section{Preliminaries}
\label{sec:preliminaries}

We write $\Nat=\{0,1,2,\ldots\}$ and use bold letters for elements of
$\Nat^k$.  The standard generators are $\bm e_1,\ldots,\bm e_k$.  All
categories and groupoids are small.  In the operator-algebraic sections they
are also countable and discrete.

\subsection{Categories with a size functor}

For an arrow $c$ in a category $\Ccat$, its range and source are denoted by
$r(c)$ and $s(c)$, respectively.  Thus $cd$ is defined exactly when
$s(c)=r(d)$.  We identify objects with identity arrows and write $\Ccat^0$ for
their set.  The wide subgroupoid of invertible arrows is $\Ccat^\times$.

\begin{definition}
\label{def:generalized-graph}
A \emph{size functor} is a functor $\ell\colon\Ccat\to\Nat^k$ such that
$\ell^{-1}(\bm 0)=\Ccat^\times$.  It has the \emph{weak factorization
property} if, whenever $\ell(c)=\bm m+\bm n$, there are arrows $a,b$ such
that
\[
 c=ab,\qquad \ell(a)=\bm m,\qquad \ell(b)=\bm n,
\]
and, for every other such factorization $c=a'b'$, there is a
$g\in\Ccat^\times$ for which
\[
                a'=ag,\qquad b'=g^{-1}b.
\]
A category equipped with such a functor is a \emph{generalized higher-rank
$k$-graph}.
\end{definition}

This is Lawson--Vdovina's definition
\cite[Section~3]{LawsonVdovina2022}, with $s$ in place of their domain map
$d$.  The definition does not assert that the connecting invertible arrow is
unique.  An ordinary $k$-graph is the special case in which the only
invertible arrows are identities.

A category is \emph{left cancellative} if $ca=cb$ implies $a=b$ whenever the
products are defined.  For a generalized higher-rank graph, this hypothesis
has an equivalent formulation on the homogeneous bisets.

\begin{lemma}
\label{lem:right-action-free}
Let $(\Ccat,\ell)$ be a generalized higher-rank graph and
$\G=\Ccat^\times$.  Then $\Ccat$ is left cancellative if and only if right
multiplication makes each set $\Ccat^{\bm n}=\ell^{-1}(\bm n)$ a free right
$\G$-set.
\end{lemma}

\begin{proof}
If $cg=ch$, left cancellation gives $g=h$, so every right action is free.
Conversely, suppose that
all the right actions are free and that $ca=cb$.  Additivity of the size
functor and cancellation in $\Nat^k$ give $\ell(a)=\ell(b)$.  Apply weak
factorization to the two factorizations $c\cdot a$ and $c\cdot b$ of the same
arrow.  There is a $g\in\G$ such that
\[
                         c=cg,\qquad b=g^{-1}a.
\]
Freeness of the right action on the fiber containing $c$ forces
$g=s(c)$, and hence $a=b$.
\end{proof}

\subsection{Groupoid bisets and balanced products}

Let $\G$ be a groupoid.  A $\G$--$\G$-biset $Z$ has anchor maps
$r_Z,s_Z\colon Z\to\G^0$, a left action defined when $s(g)=r_Z(z)$, and a
right action defined when $s_Z(z)=r(g)$.  The anchor maps transform as they do
under multiplication of arrows, and the two actions commute.  We suppress
the subscripts on the anchors.

If $Z$ and $W$ are $\G$--$\G$-bisets, let
\[
 Z\,{}_{s}\!\mathbin{*}_{r}W
   =\{(z,w)\in Z\times W:s(z)=r(w)\}.
\]
There is a partially defined diagonal action
\[
        (z,w)\cdot g=(zg,g^{-1}w).
\]
The \emph{balanced product} $Z*_{\G}W$ is the orbit set for this action; its
elements are written $[z,w]$.  Equivalently, it is the quotient generated by
$(zg,w)\sim(z,gw)$.  It is again a $\G$--$\G$-biset.

\begin{definition}
\label{def:biset-product-system}
A \emph{normalized product system of $\G$--$\G$-bisets over $\Nat^k$}
consists of bisets $Z_{\bm n}$, with $Z_{\bm 0}=\G$, and biset bijections
\[
 \mu_{\bm m,\bm n}\colon
 Z_{\bm m}*_{\G}Z_{\bm n}\longrightarrow Z_{\bm m+\bm n}.
\]
The maps with a zero index are the left and right actions of $\G$, and the
maps are associative in the sense that
\begin{equation}
\label{eq:biset-associativity}
 \mu_{\bm m+\bm n,\bm p}
   \big[\mu_{\bm m,\bm n}[a,b],c\big]
 =
 \mu_{\bm m,\bm n+\bm p}
   \big[a,\mu_{\bm n,\bm p}[b,c]\big]
\end{equation}
for every composable triple.
\end{definition}

\subsection{\texorpdfstring{$C^*$}{C-star}-correspondences}

Let $A$ be a $C^*$-algebra.  An $A$-correspondence is a right Hilbert
$A$-module $X$ together with a nondegenerate $*$-homomorphism
$A\to\adjointables(X)$.  A product system over $\Nat^k$ consists of
$A$-correspondences $X_{\bm n}$, with $X_{\bm 0}=A$, and associative unitary
correspondence maps
\[
 U_{\bm m,\bm n}\colon
 X_{\bm m}\otimes_A X_{\bm n}\longrightarrow X_{\bm m+\bm n}
\]
whose zero-degree maps are the module actions.  We use the standard order on
$\Nat^k$ and write $\bm m\vee\bm n$ for the coordinatewise maximum.  The
system is \emph{compactly aligned} if the product of the canonical
amplifications to $X_{\bm m\vee\bm n}$ of any
$S\in\compacts(X_{\bm m})$ and $T\in\compacts(X_{\bm n})$ belongs to
$\compacts(X_{\bm m\vee\bm n})$; see
\cite[Definition~5.7]{Fowler2002}.
 \section{Weak factorization as balanced factorization}
\label{sec:biset-product-systems}

Let $(\Ccat,\ell)$ be a category with a size functor, and put
$\G=\Ccat^\times$.  For $\bm n\in\Nat^k$, set
\[
                    Z_{\bm n}=\Ccat^{\bm n}=\ell^{-1}(\bm n).
\]
Left and right multiplication by invertible arrows make $Z_{\bm n}$ a
$\G$--$\G$-biset with anchors $r$ and $s$.  These actions need not be free.

\begin{proposition}
\label{prop:multiplication-balanced}
For $\bm m,\bm n\in\Nat^k$, multiplication in $\Ccat$ induces a biset map
\begin{equation}
\label{eq:mu-mn}
 \mu_{\bm m,\bm n}\colon
 Z_{\bm m}*_{\G}Z_{\bm n}\longrightarrow Z_{\bm m+\bm n},
 \qquad \mu_{\bm m,\bm n}[a,b]=ab.
\end{equation}
The size functor has the weak factorization property if and only if every map
in \eqref{eq:mu-mn} is bijective.
\end{proposition}

\begin{proof}
If $(ag,b)$ and $(a,gb)$ are defined, associativity gives
$(ag)b=a(gb)$.  Thus \eqref{eq:mu-mn} is well defined.  It is equivariant for
the exterior left and right actions.

Suppose first that $\ell$ has the weak factorization property.  Existence of a
factorization with degrees $\bm m$ and $\bm n$ proves surjectivity.  If
$ab=a'b'$ for two such factorizations, compare both with a factorization
specified by Definition~\ref{def:generalized-graph}.  Composing the two
resulting invertible middle arrows gives a $g\in\G$ such that
\[
                         a'=ag,\qquad b'=g^{-1}b.
\]
Thus $(a,b)$ and $(a',b')$ lie in the same diagonal $\G$-orbit, so
$[a,b]=[a',b']$.  Hence \eqref{eq:mu-mn} is injective.

Conversely, suppose all the maps \eqref{eq:mu-mn} are bijective.  Their
surjectivity gives the required factorizations.  If $ab=a'b'$ with the same
prescribed degrees, injectivity gives $[a,b]=[a',b']$.  The balanced
equivalence relation is the orbit relation for the diagonal groupoid action,
so there is a single $g\in\G$ such that
$(a',b')=(ag,g^{-1}b)$.  This is the weak factorization property.
\end{proof}

\begin{theorem}
\label{thm:wfp-biset-system}
Let $(\Ccat,\ell)$ be a generalized higher-rank $k$-graph.  Its homogeneous
fibers, together with the maps \eqref{eq:mu-mn}, form a normalized product
system of $\G$--$\G$-bisets over $\Nat^k$.

Conversely, let $Z=\{Z_{\bm n}:\bm n\in\Nat^k\}$ be a normalized product
system of $\G$--$\G$-bisets.  There is a generalized higher-rank graph
\[
                       \Ccat_Z=\bigsqcup_{\bm n\in\Nat^k}Z_{\bm n}
\]
with size functor $\ell(z)=\bm n$ for $z\in Z_{\bm n}$ and multiplication
\begin{equation}
\label{eq:reconstructed-multiplication}
                         zw=\mu_{\bm m,\bm n}[z,w]
\end{equation}
for composable $z\in Z_{\bm m}$ and $w\in Z_{\bm n}$.  Its invertible
subgroupoid is $Z_{\bm 0}=\G$.
\end{theorem}

\begin{proof}
For a generalized higher-rank graph, Proposition
\ref{prop:multiplication-balanced} gives the required bijections.  The maps
with a zero index are multiplication by elements of $\G$.  For a composable
triple $a,b,c$, the two sides of \eqref{eq:biset-associativity} are
$(ab)c$ and $a(bc)$, respectively, and hence agree.

Now start with $Z$.  Use the common anchor set $\G^0$ as the object set of
$\Ccat_Z$.  The unit conditions in Definition
\ref{def:biset-product-system} show that the units act as identities in
\eqref{eq:reconstructed-multiplication}.  The coherence relation
\eqref{eq:biset-associativity} proves associativity, so $\Ccat_Z$ is a
category.  The grading is additive by construction.  Every element of
$Z_{\bm 0}=\G$ is invertible.  Conversely, if an element of degree $\bm n$
has an inverse of degree $\bm p$, then $\bm n+\bm p=\bm 0$ in $\Nat^k$;
hence $\bm n=\bm 0$.  Thus the invertible subgroupoid is exactly $\G$.
Proposition~\ref{prop:multiplication-balanced}, applied to the bijections
$\mu_{\bm m,\bm n}$, gives the weak factorization property.
\end{proof}

The preceding constructions also account for morphisms.  Fix $\G$.  A
\emph{$\G$-fixed morphism} of generalized higher-rank graphs is a
degree-preserving functor whose restriction to $\G$ is the identity.  A
morphism of normalized biset product systems is a family of equivariant maps
$\theta_{\bm n}\colon Z_{\bm n}\to W_{\bm n}$ with
$\theta_{\bm 0}=\id_{\G}$ and
\[
 \theta_{\bm m+\bm n}\bigl(\mu^Z_{\bm m,\bm n}[z,w]\bigr)
 =\mu^W_{\bm m,\bm n}
       [\theta_{\bm m}(z),\theta_{\bm n}(w)].
\]

\begin{corollary}
\label{cor:categorical-isomorphism}
For a fixed groupoid $\G$, taking homogeneous fibers and taking their disjoint
union are mutually inverse functors between generalized higher-rank
$k$-graphs with invertible subgroupoid $\G$ and normalized
$\Nat^k$-product systems of $\G$--$\G$-bisets.
\end{corollary}

\begin{proof}
A $\G$-fixed degree-preserving functor restricts to equivariant maps on the
fibers, and functoriality is precisely compatibility with the maps $\mu$.
Conversely, the disjoint union of a compatible family $\theta_{\bm n}$ is a
functor by \eqref{eq:reconstructed-multiplication}.  On objects and arrows,
both composite constructions are identities.
\end{proof}

\begin{remark}
The statement of Corollary~\ref{cor:categorical-isomorphism} fixes the
degree-zero groupoid in order to keep the morphisms elementary.  Allowing a
functor between degree-zero groupoids gives the corresponding fibred version,
with equivariance understood relative to that functor.
\end{remark}
 \section{Linearization over the full groupoid algebra}
\label{sec:linearization}

Throughout this section $(\Ccat,\ell)$ is a countable, left-cancellative
generalized higher-rank $k$-graph, $\G=\Ccat^\times$, and
$A=\cstar{\G}$.  Write $u_g$ for the canonical partial isometry associated to
$g\in\G$ and $p_v=u_v$ for $v\in\G^0$.  Thus
\[
 u_g^*=u_{g^{-1}},\qquad
 u_g u_h=
 \begin{cases}
   u_{gh},&s(g)=r(h),\\
   0,&\text{otherwise}.
 \end{cases}
\]
The algebraic span $C_c(\G)=\operatorname{span}\{u_g:g\in\G\}$ is a dense
$*$-subalgebra of $A$.

\subsection{The homogeneous correspondences}

For $\bm n\in\Nat^k$, let $X_{\bm n}^{\mathrm{alg}}$ be the complex span of
symbols $\xi_c$ indexed by $c\in\Ccat^{\bm n}$.  On generators of
$C_c(\G)$ define
\begin{equation}
\label{eq:right-module-action}
 \xi_c u_g=
 \begin{cases}
   \xi_{cg},&s(c)=r(g),\\
   0,&\text{otherwise},
 \end{cases}
\end{equation}
and
\begin{equation}
\label{eq:intrinsic-inner-product}
 \ip{\xi_c}{\xi_d}_A=
 \begin{cases}
   u_g,&d=cg\text{ for some }g\in\G,\\
   0,&d\notin c\G.
 \end{cases}
\end{equation}
The $g$ in \eqref{eq:intrinsic-inner-product} is unique by
Lemma~\ref{lem:right-action-free}.  Extend the action linearly, and extend the
sesquilinear form linearly in the second variable and conjugate linearly in
the first.

\begin{proposition}
\label{prop:coordinate-module}
Choose a set $T_{\bm n}\subseteq\Ccat^{\bm n}$ containing one representative
of every right $\G$-orbit.  The preceding formulas make
$X_{\bm n}^{\mathrm{alg}}$ a pre-Hilbert right $C_c(\G)$-module.  Its right
action extends to $A$ on the Hilbert-module completion $X_{\bm n}$, and there
is a unitary right-module identification
\begin{equation}
\label{eq:coordinate-module}
 X_{\bm n}\cong
 \bigoplus_{t\in T_{\bm n}}p_{s(t)}A,
 \qquad
 \xi_{tg}\longmapsto e_tu_g,
\end{equation}
where $e_t$ is the standard vector in the $t$-summand.  In particular,
$X_{\bm 0}\cong A$ by $\xi_g\mapsto u_g$.
\end{proposition}

\begin{proof}
Every $c\in\Ccat^{\bm n}$ has an expression $c=tg$ with
$t\in T_{\bm n}$.  The expression is unique because the right action is
free.  Hence
\[
 X_{\bm n}^{\mathrm{alg}}
 \longrightarrow \bigoplus_{t\in T_{\bm n}}^{\mathrm{alg}}
                       p_{s(t)}C_c(\G),
 \qquad \xi_{tg}\longmapsto e_tu_g,
\]
is a linear bijection.  It transports the right $C_c(\G)$-action to ordinary
coordinatewise multiplication.  If
$c=tg$ and $d=th$, then
\[
 \ip{e_tu_g}{e_tu_h}
 =u_g^*p_{s(t)}u_h=u_{g^{-1}h}
 =\ip{\xi_c}{\xi_d}.
\]
Vectors belonging to distinct summands are orthogonal, as are symbols in
distinct right orbits.  Thus the map also transports the sesquilinear form to
the standard positive-definite inner product.  This proves the pre-Hilbert
module assertions.  The algebraic direct sum is dense in
$\bigoplus_{t\in T_{\bm n}}p_{s(t)}A$, where multiplication gives the required
extension of the right action to $A$.  Completion now gives the unitary
\eqref{eq:coordinate-module}.

For $\bm n=\bm 0$, take $T_{\bm 0}=\G^0$.  The orthogonal sum of the modules
$p_vA$ maps isometrically onto the closure of $\sum_vp_vA$, which is $A$.
\end{proof}

For $h\in\G$, define on elementary vectors
\begin{equation}
\label{eq:left-action}
 L_h\xi_c=
 \begin{cases}
   \xi_{hc},&s(h)=r(c),\\
   0,&\text{otherwise}.
 \end{cases}
\end{equation}

\begin{proposition}
\label{prop:left-action}
The operators $L_h$ are adjointable, $L_h^*=L_{h^{-1}}$, and the assignment
$u_h\mapsto L_h$ extends to a nondegenerate $*$-homomorphism
\[
                   \varphi_{\bm n}\colon A\longrightarrow
                   \adjointables(X_{\bm n}).
\]
Consequently every $X_{\bm n}$ is an $A$-correspondence.  The homomorphism
$\varphi_{\bm n}$ need not be injective.
\end{proposition}

\begin{proof}
Use the coordinates in Proposition~\ref{prop:coordinate-module}.  If
$s(h)=r(t)$, there are unique $t'\in T_{\bm n}$ and $g\in\G$ such that
$ht=t'g$, and
\begin{equation}
\label{eq:monomial-left-action}
                         L_he_t=e_{t'}u_g.
\end{equation}
Left multiplication by $h$ bijects its domain of right orbits with the domain
of left multiplication by $h^{-1}$.  Formula
\eqref{eq:monomial-left-action} therefore defines an adjointable partial
isometry with adjoint $L_{h^{-1}}$.  The category laws give
$L_gL_h=L_{gh}$ for composable $g,h$ and zero otherwise.  The universal
property of the full groupoid $C^*$-algebra gives the homomorphism
$\varphi_{\bm n}$.  Finite sums of vertex projections form an approximate
identity for $A$, and $L_{r(c)}\xi_c=\xi_c$; hence the left action is
nondegenerate.
\end{proof}

Although a transversal was used to verify the Hilbert-module axioms, both
\eqref{eq:right-module-action} and \eqref{eq:intrinsic-inner-product} are
intrinsic.  A different transversal only changes the coordinate vectors by
canonical groupoid partial isometries.

\subsection{Multiplication unitaries}

For $a\in\Ccat^{\bm m}$ and $b\in\Ccat^{\bm n}$ define
\begin{equation}
\label{eq:U-on-generators}
 U_{\bm m,\bm n}(\xi_a\otimes\xi_b)=
 \begin{cases}
   \xi_{ab},&s(a)=r(b),\\
   0,&\text{otherwise}.
 \end{cases}
\end{equation}
A noncomposable elementary tensor is already zero: indeed,
$\ip{\xi_a}{\xi_a}=p_{s(a)}$, and
$\varphi_{\bm n}(p_{s(a)})\xi_b=0$ unless $s(a)=r(b)$.

\begin{theorem}
\label{thm:linearization}
For every $\bm m,\bm n\in\Nat^k$, formula
\eqref{eq:U-on-generators} extends to a unitary correspondence map
\[
 U_{\bm m,\bm n}\colon
 X_{\bm m}\otimes_A X_{\bm n}\longrightarrow X_{\bm m+\bm n}.
\]
These unitaries are associative and make
$X(\Ccat)=\{X_{\bm n}:\bm n\in\Nat^k\}$ a product system of
$A$-correspondences.
\end{theorem}

\begin{proof}
The balancing relation follows directly from associativity in $\Ccat$:
\[
 U_{\bm m,\bm n}(\xi_{ag}\otimes\xi_b)
 =\xi_{(ag)b}
 =\xi_{a(gb)}
 =U_{\bm m,\bm n}(\xi_a\otimes L_g\xi_b).
\]
The same computation proves compatibility with the exterior left and right
$A$-actions.

For the inner products, take
composable pairs $(a,b)$ and $(a',b')$ of degrees $(\bm m,\bm n)$.  By the
definition of the tensor-product inner product,
\begin{align*}
 &\ip{\xi_a\otimes\xi_b}{\xi_{a'}\otimes\xi_{b'}}_A\\
 &\qquad=
 \begin{cases}
   u_h,&a'=ag\text{ and }bh=gb'
          \text{ for some }g,h\in\G,\\
   0,&\text{otherwise}.
 \end{cases}
\end{align*}
In the nonzero case,
\[
                         (ab)h=a(gb')=a'b',
\]
so the displayed value is
$\ip{\xi_{ab}}{\xi_{a'b'}}_A$.  Conversely, suppose this latter inner product
is $u_h$.  Then $(ab)h=a'b'$.  Apply weak factorization to the two
factorizations $a(bh)$ and $a'b'$ of the same arrow at degrees
$\bm m,\bm n$.  There is a $g\in\G$ such that
\[
                         a'=ag,\qquad b'=g^{-1}bh,
\]
or equivalently $bh=gb'$.  Thus the tensor-product inner product is also
$u_h$.  This proves that \eqref{eq:U-on-generators} is isometric.

If $c\in\Ccat^{\bm m+\bm n}$, weak factorization gives $c=ab$ with the
prescribed degrees.  Hence $\xi_c$ lies in the range.  The isometry has closed
range containing a dense spanning set, so it is unitary.  Finally, both ways
of multiplying a composable elementary triple map it to $\xi_{abc}$.
Associativity follows by density.  Proposition~\ref{prop:coordinate-module}
identifies the zero fiber and zero-degree multiplication with the standard
correspondence $A$ and its module actions.
\end{proof}

\subsection{A sufficient condition for compact alignment}

Call $(\Ccat,\ell)$ \emph{orbit-finite} if
$\Ccat^{\bm n}/\G$ is finite for every nonzero $\bm n\in\Nat^k$.

\begin{corollary}
\label{cor:compact-alignment}
If $(\Ccat,\ell)$ is orbit-finite, then
$\varphi_{\bm n}(A)\subseteq\compacts(X_{\bm n})$ for every nonzero
$\bm n$, and $X(\Ccat)$ is compactly aligned.
\end{corollary}

\begin{proof}
Choose a finite transversal $T_{\bm n}$.  In the coordinates
\eqref{eq:coordinate-module},
\[
                  1_{\adjointables(X_{\bm n})}
                  =\sum_{t\in T_{\bm n}}\Theta_{e_t,e_t}.
\]
Thus the identity operator is compact.  Since the compact operators form an
ideal in the adjointable operators,
$\adjointables(X_{\bm n})=\compacts(X_{\bm n})$.  This proves the assertion
about the left action.

Let $S\in\compacts(X_{\bm m})$ and
$T\in\compacts(X_{\bm n})$.  If
$\bm m\vee\bm n\ne\bm 0$, their amplified product is adjointable on
$X_{\bm m\vee\bm n}$ and is therefore compact by the preceding paragraph.
If $\bm m\vee\bm n=\bm 0$, then $\bm m=\bm n=\bm 0$ and the assertion is
the equality $\compacts(A_A)=A$.  Hence the system is compactly aligned.
\end{proof}
 \section{Degree-one coordinates and strict splittings}
\label{sec:coordinates}

In the factorization-unitaries subsection we retain the standing assumptions of
Section~\ref{sec:linearization}: $(\Ccat,\ell)$ is countable and left
cancellative, $\G=\Ccat^\times$, $A=C^*(\G)$, and $X(\Ccat)$ is the product
system of Theorem~\ref{thm:linearization}.

\subsection{Factorization unitaries}

Put $E_i=X_{\bm e_i}$.  For distinct $i,j$, define
\begin{equation}
\label{eq:chi-ij}
 \chi_{ij}
 =U_{\bm e_j,\bm e_i}^{-1}U_{\bm e_i,\bm e_j}
 \colon E_i\otimes_AE_j\longrightarrow E_j\otimes_AE_i.
\end{equation}
These maps are the correspondence-level commuting squares.

\begin{proposition}
\label{prop:degree-one-coherence}
The product system $X(\Ccat)$ is determined, up to isomorphism, by the
correspondences $E_1,\ldots,E_k$ and the unitaries $\chi_{ij}$.  For distinct
$i,j,l$, the latter satisfy the braid relation
\begin{align}
\label{eq:braid-relation}
 &(\chi_{jl}\otimes 1_{E_i})(1_{E_j}\otimes\chi_{il})
       (\chi_{ij}\otimes1_{E_l}) \\
 &\qquad=(1_{E_l}\otimes\chi_{ij})(\chi_{il}\otimes1_{E_j})
       (1_{E_i}\otimes\chi_{jl}),
\end{align}
with the canonical associativity identifications suppressed.  If orbit
transversals are chosen in the degree-one fibers, each $\chi_{ij}$ is a
monomial unitary whose nonzero coefficients are canonical elements of
$\cstar{\G}$.  Moreover, $\chi_{ji}=\chi_{ij}^{-1}$, and swaps acting on
disjoint tensor positions commute.
\end{proposition}

\begin{proof}
Iterating the multiplication unitaries identifies every $X_{\bm n}$ with
\[
 E_1^{\otimes n_1}\otimes_A\cdots\otimes_AE_k^{\otimes n_k}.
\]
Changing the order of adjacent colors is implemented by
\eqref{eq:chi-ij}.  Associativity of the maps $U$ says that the two ways of
reversing three distinct adjacent colors agree, which is exactly
\eqref{eq:braid-relation}.  The identity
$\chi_{ji}=\chi_{ij}^{-1}$ follows at once from \eqref{eq:chi-ij}, while
swaps on disjoint tensor positions commute by functoriality of the interior
tensor product.  These are the inverse, braid, and far-commutativity Coxeter
relations.  Consequently, the unitary associated to a reordering is
independent of the chosen adjacent-swap word, and the stated data recover the
multiplication maps.

For the coordinate assertion, take composable
$x\in T_{\bm e_i}$ and $y\in T_{\bm e_j}$.  Refactor $xy=ab$ with
$\ell(a)=\bm e_j$ and $\ell(b)=\bm e_i$.  Write uniquely
\[
                         a=y'g,\qquad gb=x'h
\]
with $y'\in T_{\bm e_j}$, $x'\in T_{\bm e_i}$, and $g,h\in\G$.  Balancing
gives
\[
 \chi_{ij}(e_x\otimes e_y)
   =(e_{y'}\otimes e_{x'})u_h.
\]
Thus every basis vector is sent to a basis vector times one groupoid partial
isometry.  The formula is independent of the intermediate factorization by
weak factorization and balancing.
\end{proof}

For $k=1$ there are no commutation maps and
$X_n\cong X_1^{\otimes_A n}$.  In the one-object case arising from a
self-similar group $G$ acting on the free monoid on a finite alphabet $Y$, one
has
\[
 X_1=\bigoplus_{y\in Y}C^*(G),
 \qquad u_g e_y=e_{g\cdot y}u_{g|_y}.
\]
This is the full-$C^*(G)$ completion of Nekrashevych's algebraic
self-similarity bimodule and its matrix-valued linear recursion
\cite[pp.~227--228]{Nekrashevych2004}.  A general rank-one generalized
higher-rank graph need not come from a one-object free monoid, so the
identification with Nekrashevych's setting requires these additional
hypotheses.

Every Toeplitz representation satisfies the category relations below.

\begin{proposition}
\label{prop:representation-relations}
Let $\psi$ be a Toeplitz representation of $X(\Ccat)$ in a $C^*$-algebra.
For $g\in\G$ and $c\in\Ccat^{\bm n}$ set
\[
                   U_g=\psi_{\bm 0}(u_g),
                   \qquad S_c=\psi_{\bm n}(\xi_c).
\]
Then $U$ is a groupoid representation and
\begin{align*}
 S_aS_b&=S_{ab} &&\text{if }s(a)=r(b),\\
 U_gS_c&=S_{gc} &&\text{if }s(g)=r(c),\\
 S_cU_g&=S_{cg} &&\text{if }s(c)=r(g).
\end{align*}
The corresponding product is zero when the indicated arrows are not
composable.  For $c,d$ of the same degree,
\[
 S_c^*S_d=
 \begin{cases}
   U_g,&d=cg\text{ for }g\in\G,\\
   0,&d\notin c\G.
 \end{cases}
\]
\end{proposition}

\begin{proof}
The assertions are respectively the multiplicativity, bimodule, and
inner-product axioms for a Toeplitz representation, applied to
\eqref{eq:U-on-generators}, \eqref{eq:left-action},
\eqref{eq:right-module-action}, and \eqref{eq:intrinsic-inner-product}.
\end{proof}

\subsection{The \texorpdfstring{$R$}{R}-condition as strictifiability}

The result in this subsection is categorical: here $(\Ccat,\ell)$ may be any
generalized higher-rank graph, without countability or cancellation
assumptions, and $\G=\Ccat^\times$.

An arrow of $\Ccat$ is an \emph{atom} if it is noninvertible and every
factorization of it has an invertible factor.  In a generalized higher-rank
graph, the atoms are precisely the arrows whose degrees are among
$\bm e_1,\ldots,\bm e_k$ \cite[Lemma~3.1]{LawsonVdovina2022}.
Lawson--Vdovina formulate the $R$-condition using representatives for
generators of maximal principal right ideals; their Lemmas~2.1 and~2.3
identify these representatives, up to right multiplication by invertible
arrows, with the atoms used here.  Choose an atomic right-orbit transversal
$D$, and let $\Lambda=\langle D\rangle$ be the wide subcategory generated by
$D$.  Their $R$-condition, relative to this choice, is
\begin{equation}
\label{eq:r-condition}
 \lambda=\mu g,\quad \lambda,\mu\in\Lambda,\quad g\in\G
 \quad\Longrightarrow\quad
 \lambda=\mu\text{ and }g=s(\mu).
\end{equation}
The condition is imposed on the generated subcategory, not only on $D$.

\begin{definition}
\label{def:strict-splitting}
A \emph{strict splitting} of the biset product system
$\{\Ccat^{\bm n}\}$ is a family of subsets
$\Lambda^{\bm n}\subseteq\Ccat^{\bm n}$ such that
$\Lambda^{\bm 0}=\Ccat^0$, multiplication maps composable elements of
$\Lambda^{\bm m}\times\Lambda^{\bm n}$ into
$\Lambda^{\bm m+\bm n}$, and
\begin{equation}
\label{eq:strict-normal-form}
 \Lambda^{\bm n}\,{}_{s}\!\mathbin{*}_{r}\G
 \longrightarrow\Ccat^{\bm n},
 \qquad (\lambda,g)\longmapsto\lambda g
\end{equation}
is a bijection for every $\bm n$.
\end{definition}

\begin{theorem}
\label{thm:r-condition-splitting}
Let $D$ be an atomic right-orbit transversal and
$\Lambda=\langle D\rangle$.  The following are equivalent.
\begin{enumerate}[(i)]
\item The $R$-condition \eqref{eq:r-condition} holds.
\item The sets $\Lambda^{\bm n}=\Lambda\cap\Ccat^{\bm n}$ form a strict
      splitting.
\item The category $\Lambda$ is a $k$-graph and every $c\in\Ccat$ has a
      unique normal form $c=\lambda g$ with $\lambda\in\Lambda$ and
      $g\in\G$.
\end{enumerate}
In this case $\Ccat$ is isomorphic to a size-preserving Zappa--Sz\'ep product
$\Lambda\bowtie\G$.
\end{theorem}

\begin{proof}
The implication (i)$\Rightarrow$(iii) is
\cite[Theorem~4.2]{LawsonVdovina2022}.  The unique normal form gives the
bijections \eqref{eq:strict-normal-form}, while closure under multiplication
holds because $\Lambda$ is a subcategory.  Thus (iii)$\Rightarrow$(ii).

Suppose (ii) holds.  Let
$\lambda\in\Lambda^{\bm m+\bm n}$.  Weak factorization gives
$\lambda=ab$ with degrees $\bm m,\bm n$.  Use the strict normal forms to
write $a=\alpha g$ and $gb=\beta h$, where
$\alpha\in\Lambda^{\bm m}$ and $\beta\in\Lambda^{\bm n}$.  Then
$\lambda=(\alpha\beta)h$.  Both $\lambda$ and $\alpha\beta$ belong to the
splitting, so uniqueness in \eqref{eq:strict-normal-form} gives
$\lambda=\alpha\beta$ and $h=s(\lambda)$.  Thus factorizations exist inside
$\Lambda$.  If
$\lambda=\alpha\beta=\alpha'\beta'$ are two such factorizations, weak
factorization gives
$\alpha'=\alpha g$ and $\beta'=g^{-1}\beta$.  Uniqueness of strict normal
forms forces $g$ to be an identity and then
$\alpha'=\alpha$, $\beta'=\beta$.  Hence $\Lambda$ is a $k$-graph and (iii)
holds.  Finally, uniqueness of the two normal forms for
$\lambda=\mu g$ proves (i).

When these conditions hold, factor $g\lambda$ uniquely as
\[
                         g\lambda=(g\cdot\lambda)(g|_\lambda),
 \qquad g\cdot\lambda\in\Lambda,\quad g|_\lambda\in\G.
\]
Associativity and uniqueness give the Zappa--Sz\'ep axioms, and
\[
 (\lambda,g)(\mu,h)
   =\bigl(\lambda(g\cdot\mu),(g|_\mu)h\bigr).
\]
This is the normal-form construction of
\cite[Proposition~4.4 and Theorem~4.8]{LawsonVdovina2022}.
\end{proof}

We return to the countable, left-cancellative setting of
Section~\ref{sec:linearization}.  The strict splitting gives coordinates for
the correspondence product system.  Let
$X(\Lambda^{\bm n})$ denote the graph correspondence of the directed graph
with edge set $\Lambda^{\bm n}$.  On
$X(\Lambda^{\bm n})\otimes_{C_0(\Lambda^0)}A$, define the left action on
canonical generators by
\begin{equation}
\label{eq:zappa-left-action}
 u_g(\delta_\lambda\otimes a)=
 \begin{cases}
  \delta_{g\cdot\lambda}\otimes u_{g|_\lambda}a,
       &s(g)=r(\lambda),\\
  0,&\text{otherwise}.
 \end{cases}
\end{equation}

We call a $k$-graph $\Lambda$ \emph{finite} when $\Lambda^{\bm n}$ is finite
for every $\bm n\in\Nat^k$.

\begin{theorem}
\label{thm:strict-correspondence-comparison}
Under the equivalent conditions of Theorem
\ref{thm:r-condition-splitting}, there are correspondence unitaries
\begin{equation}
\label{eq:strict-correspondence-unitary}
 W_{\bm n}\colon
 X(\Lambda^{\bm n})\otimes_{C_0(\Lambda^0)}A
 \longrightarrow X_{\bm n},
 \qquad
 W_{\bm n}(\delta_\lambda\otimes u_g)=\xi_{\lambda g}.
\end{equation}
They form an isomorphism of product systems, and multiplication in the left
side is
\begin{equation}
\label{eq:zappa-correspondence-product}
 (\delta_\lambda\otimes u_g)(\delta_\mu\otimes u_h)
 =
 \begin{cases}
 \delta_{\lambda(g\cdot\mu)}\otimes u_{(g|_\mu)h},
       &s(g)=r(\mu),\\
 0,&\text{otherwise}.
 \end{cases}
\end{equation}
If $\Lambda$ is finite, this product system is compactly aligned.  If, in
addition, the action $\G\to\operatorname{PIso}(\Lambda)$ is faithful, it is
the system constructed in
\cite[Proposition~5.1]{AfsarBrownloweRamaggeWhittaker2026}.
\end{theorem}

\begin{proof}
The strict normal form says that $\Lambda^{\bm n}$ is a right-orbit
transversal, so Proposition~\ref{prop:coordinate-module} gives
the right-Hilbert-module unitary
\eqref{eq:strict-correspondence-unitary}.  If
$g\lambda=(g\cdot\lambda)(g|_\lambda)$, the two left actions agree because
\[
 u_g\xi_\lambda=\xi_{g\lambda}
 =\xi_{g\cdot\lambda}u_{g|_\lambda}.
\]
In particular, \eqref{eq:zappa-left-action} extends to a nondegenerate
left action of $A$, and \eqref{eq:strict-correspondence-unitary} is a
correspondence unitary.
The Zappa--Sz\'ep multiplication formula gives
\eqref{eq:zappa-correspondence-product}, so the $W_{\bm n}$ intertwine all
multiplication unitaries.

If $\Lambda$ is finite, every $\Lambda^{\bm n}$ is finite and identifies with
$\Ccat^{\bm n}/\G$.  Compact alignment follows from
Corollary~\ref{cor:compact-alignment}.  Under faithfulness, the fibers, left
actions, and multiplication in \eqref{eq:zappa-correspondence-product} are
exactly those in the cited construction.
\end{proof}

\begin{remark}
The faithfulness clause in the last theorem is essential for the literal
literature comparison.  A Zappa--Sz\'ep action may act trivially on
$\Lambda$ while retaining nontrivial restriction elements.  The
$R$-condition still holds for the product category, but the action is not a
self-similar groupoid action under the faithful-action convention of
\cite[Definition~3.3]{AfsarBrownloweRamaggeWhittaker2026}.
\end{remark}
 \section{A cancellative nonsplit example}
\label{sec:twisted-example}

The following example separates left cancellation from the $R$-condition and
exhibits the groupoid coefficient carried by a commuting square.

\begin{example}
\label{ex:heisenberg}
Let
\[
                         \Ccat=\Nat^2\times\Intg
\]
be the one-object category with multiplication
\begin{equation}
\label{eq:heisenberg-product}
 (m,n,l)(p,q,r)
   =(m+p,n+q,l+r-np).
\end{equation}
Its identity is $(0,0,0)$, and define
\[
                         \ell(m,n,l)=(m,n).
\]
Then $(\Ccat,\ell)$ is a cancellative generalized higher-rank $2$-graph that
does not satisfy the $R$-condition for any atomic right-orbit transversal.
\end{example}

\begin{proof}
For a third element $(a,b,c)$, both associations of a triple product have
third coordinate
\[
                      l+r+c-np-na-qa.
\]
Thus \eqref{eq:heisenberg-product} is associative.  Equality after
multiplication on either the left or the right first forces equality of the
two degree coordinates and then of the integer coordinate, so $\Ccat$ is
cancellative.  Its invertible elements are exactly
\[
                  \G=\{(0,0,l):l\in\Intg\}\cong\Intg.
\]

Fix a decomposition $(M,N)=(m,n)+(p,q)$.  A factorization of
$(M,N,L)$ with these degrees is a pair
\[
              (m,n,l)(p,q,r),\qquad l+r-np=L.
\]
If $(l,r)$ and $(l',r')$ give two such factorizations, put $t=l'-l$ and
$g_t=(0,0,t)$.  Then
\[
 (m,n,l')=(m,n,l)g_t,
 \qquad
 (p,q,r')=g_t^{-1}(p,q,r).
\]
This proves the weak factorization property.

There is one right $\G$-orbit in each of the two atomic degree fibers.  Hence
every atomic transversal has the form
\[
                  x=(1,0,a),\qquad y=(0,1,b)
\]
for some $a,b\in\Intg$.  Formula \eqref{eq:heisenberg-product} gives
\[
                 xy=(1,1,a+b),
                 \qquad yx=(1,1,a+b-1),
\]
and therefore
\[
                         xy=(yx)g_1.
\]
Both $xy$ and $yx$ belong to the subcategory generated by the transversal,
while $g_1$ is not the identity.  Thus \eqref{eq:r-condition} fails for every
choice of $a$ and $b$.
\end{proof}

The relation $xy=(yx)g_1$ identifies the obstruction to splitting: the two
atomic products differ by a nontrivial central element.

\begin{proposition}
\label{prop:heisenberg-product-system}
Let $u$ be the canonical unitary generating
$A=C^*(\Intg)\cong C(\mathbb T)$.  For the category in
Example~\ref{ex:heisenberg}, every correspondence $X_{(m,n)}$ is the standard
rank-one correspondence $A$, with identical left and right actions.  Relative
to the orbit representative $(m,n,0)$, the multiplication unitary is
\begin{equation}
\label{eq:heisenberg-unitary}
 U_{(m,n),(p,q)}(a\otimes b)=u^{-np}ab.
\end{equation}
In particular, for standard generators $\eta_1\in X_{(1,0)}$ and
$\eta_2\in X_{(0,1)}$,
\begin{equation}
\label{eq:heisenberg-chi}
                         \chi_{12}(\eta_1\otimes \eta_2)
                         =(\eta_2\otimes \eta_1)u.
\end{equation}
\end{proposition}

\begin{proof}
Every nonempty homogeneous fiber is a free transitive right $\G$-set, so
Proposition~\ref{prop:coordinate-module} identifies its correspondence with
$A$.  The subgroup $\G$ is central in \eqref{eq:heisenberg-product}, and hence
the left action is the standard one.  The chosen representatives multiply as
\[
                 (m,n,0)(p,q,0)=(m+p,n+q,-np),
\]
which proves \eqref{eq:heisenberg-unitary}.

The map $U_{(1,0),(0,1)}$ sends $\eta_1\otimes \eta_2$ to the standard generator of
$X_{(1,1)}$, whereas $U_{(0,1),(1,0)}$ sends
$\eta_2\otimes \eta_1$ to that generator multiplied by $u^{-1}$.  Applying
\eqref{eq:chi-ij} gives \eqref{eq:heisenberg-chi}.
\end{proof}
 \section{Toeplitz and boundary algebras}
\label{sec:category-algebras}

Throughout this section $(\Ccat,\ell)$ is a countable, left-cancellative
generalized higher-rank $k$-graph, $\G=\Ccat^\times$, and
$A=C^*(\G)$.  We compare the universal algebras of $X(\Ccat)$ with the full
groupoid algebras attached to the left-cancellative category $\Ccat$.

\subsection{Finite alignment and Nica covariance}

For $a\in\Ccat$, write $a\Ccat$ for its principal right ideal.  Two arrows
$a$ and $b$ are equivalent, written $a\approx b$, when $b=ag$ for an
invertible arrow $g$; equivalently, $a\Ccat=b\Ccat$.  An arrow
$\gamma\in a\Ccat\cap b\Ccat$ is a \emph{minimal common extension} of $a$
and $b$ if $\gamma\in\eta\Ccat$ for another common extension $\eta$ implies
$\eta\approx\gamma$.  The category is \emph{finitely aligned} if every
intersection $a\Ccat\cap b\Ccat$ is empty or is a finite union of principal
right ideals.  In that case, the minimal common extensions have finitely many
$\approx$-classes
\cite[Definitions~3.1--3.2 and Lemma~3.3]{Spielberg2020}; let
$\MCE_{\G}(a,b)$ denote a set of representatives of those classes.

The grading determines the degree of every minimal common extension.

\begin{lemma}
\label{lem:mce-degree}
Let $a\in\Ccat^{\bm m}$ and $b\in\Ccat^{\bm n}$.  A common extension of
$a$ and $b$ is minimal if and only if its degree is
$\bm m\vee\bm n$.  Consequently, orbit-finiteness implies finite alignment.
\end{lemma}

\begin{proof}
Every common extension has degree at least $\bm m\vee\bm n$.  Let
$\gamma$ be one, put $\bm l=\bm m\vee\bm n$, and use weak factorization to
write $\gamma=\eta\zeta$ with $\ell(\eta)=\bm l$.  The prefix $\eta$ is
also a common extension.  If $\gamma=ax$, factor $x=x_1x_2$ with
$\ell(x_1)=\bm l-\bm m$.  Comparing
$\eta\zeta=(ax_1)x_2$ at degrees $\bm l$ and
$\ell(\gamma)-\bm l$ gives an invertible $g$ such that
$\eta=(ax_1)g$.  Thus $\eta\in a\Ccat$, and the same argument with $b$
gives $\eta\in b\Ccat$.  If $\gamma$ is minimal, then
$\eta\approx\gamma$, so their degrees agree and
$\ell(\gamma)=\bm l$.

Conversely, suppose $\ell(\gamma)=\bm l$ and
$\gamma\in\eta\Ccat$ for a common extension $\eta$.  The two degree
inequalities force $\ell(\eta)=\bm l$, so the factor from $\eta$ to
$\gamma$ has degree zero and is invertible.  Hence
$\eta\approx\gamma$.

If $\bm l=\bm 0$, then $a$ and $b$ are invertible, and their principal right
ideals are either disjoint or equal.  Suppose therefore that
$\bm l\ne\bm 0$ and that $\Ccat$ is orbit-finite.  The minimal common extensions of $a$ and $b$
belong to the single fiber $\Ccat^{\bm l}$ and therefore have only finitely
many right $\G$-orbits.  Every common extension has a prefix of degree
$\bm l$ that is again common by the first paragraph, so those finitely many
principal ideals cover $a\Ccat\cap b\Ccat$.
\end{proof}

For $\bm p\leq\bm l$, write
\[
 \iota_{\bm p}^{\bm l}(S)
 =U_{\bm p,\bm l-\bm p}
       (S\otimes 1)U_{\bm p,\bm l-\bm p}^*,
 \qquad S\in\adjointables(X_{\bm p}).
\]

\begin{proposition}
\label{prop:finite-alignment-compacts}
If $\Ccat$ is finitely aligned, then $X(\Ccat)$ is compactly aligned.  More
precisely, let $a,b\in\Ccat^{\bm p}$ with $s(a)=s(b)$ and let
$c,d\in\Ccat^{\bm q}$ with $s(c)=s(d)$.  For each
$\gamma\in\MCE_{\G}(b,c)$, write uniquely
\[
                   \gamma=b\beta_\gamma=c\delta_\gamma.
\]
Then, with $\bm l=\bm p\vee\bm q$,
\begin{equation}
\label{eq:rank-one-nica-product}
 \iota_{\bm p}^{\bm l}(\Theta_{\xi_a,\xi_b})
 \iota_{\bm q}^{\bm l}(\Theta_{\xi_c,\xi_d})
 =\sum_{\gamma\in\MCE_{\G}(b,c)}
   \Theta_{\xi_{a\beta_\gamma},\xi_{d\delta_\gamma}}.
\end{equation}
The sum is zero when there is no common extension.
\end{proposition}

\begin{proof}
The source assumptions make the displayed rank-one operators and all the
products $a\beta_\gamma$ and $d\delta_\gamma$ nonzero.  Rank-one operators
with unequal source projections are zero, so operators of the displayed form
span the compact operators.

We check \eqref{eq:rank-one-nica-product} on the elementary spanning vectors
of $X_{\bm l}$.  The amplification of
$\Theta_{\xi_c,\xi_d}$ maps an extension
$\xi_{d\varepsilon}$ to $\xi_{c\varepsilon}$ and vanishes off the right
$\G$-orbits of such extensions.  The amplification of
$\Theta_{\xi_a,\xi_b}$ can then be nonzero exactly when
$c\varepsilon$ is a common extension of $b$ and $c$.  Its degree is
$\bm l$, so Lemma~\ref{lem:mce-degree} gives unique
$\gamma\in\MCE_{\G}(b,c)$ and $h\in\G$ such that
\[
                    c\varepsilon=\gamma h
                    =b\beta_\gamma h.
\]
Left cancellation in $c\varepsilon=c\delta_\gamma h$ gives
$\varepsilon=\delta_\gamma h$.  Thus the left side maps
$\xi_{d\delta_\gamma h}$ to $\xi_{a\beta_\gamma h}$, and it vanishes on
the orthogonal complement of these right orbits.  The corresponding summand
on the right has exactly the same action because
\[
 \Theta_{\xi_{a\beta_\gamma},\xi_{d\delta_\gamma}}
       (\xi_{d\delta_\gamma h})
 =\xi_{a\beta_\gamma}u_h.
\]

Changing the representative from $\gamma$ to $\gamma g$ replaces both
vectors in the rank-one operator by right multiplication by $u_g$.  Since
their common right support is $p_{r(g)}$,
\[
 \Theta_{\xi_{a\beta_\gamma}u_g,
          \xi_{d\delta_\gamma}u_g}
 =\Theta_{\xi_{a\beta_\gamma},\xi_{d\delta_\gamma}}.
\]
Hence the formula is independent of the transversal.  Its right side is a
finite sum of compact operators.  Approximation by finite sums of rank-one
operators proves compact alignment.
\end{proof}

For a Toeplitz representation $\psi$ of $X(\Ccat)$, let
$\psi^{(\bm n)}(\Theta_{x,y})=\psi_{\bm n}(x)\psi_{\bm n}(y)^*$.  It is
\emph{Nica covariant} when
\[
 \psi^{(\bm p)}(S)\psi^{(\bm q)}(T)
 =\psi^{(\bm p\vee\bm q)}
   \left(\iota_{\bm p}^{\bm p\vee\bm q}(S)
         \iota_{\bm q}^{\bm p\vee\bm q}(T)\right)
\]
for compact $S$ and $T$.  Write $\NT_{X(\Ccat)}$ for the universal
Nica--Toeplitz algebra.

Spielberg associates to every left-cancellative small category a groupoid
$\mathscr G_2(\Ccat)$ and defines
\[
             \Tcal(\Ccat)=C^*(\mathscr G_2(\Ccat))
\]
\cite[Definition~5.13]{Spielberg2020}.  The use of $\mathscr G_2$, rather
than the germ groupoid $\mathscr G_1$, retains the invertible arrows.

\begin{theorem}
\label{thm:nica-spielberg}
If $\Ccat$ is finitely aligned, there is a canonical isomorphism
\begin{equation}
\label{eq:nica-spielberg-isomorphism}
       \NT_{X(\Ccat)}\cong C^*(\mathscr G_2(\Ccat))
\end{equation}
that sends the universal image of $\xi_c$ to Spielberg's arrow generator
$T_c$ for every $c\in\Ccat$.
\end{theorem}

\begin{proof}
Let $j$ be the universal Nica-covariant representation and put
$S_c=j_{\ell(c)}(\xi_c)$.  Proposition
\ref{prop:representation-relations} gives multiplication and source
relations.  If $a\in\Ccat^{\bm p}$ and $b\in\Ccat^{\bm q}$, Nica
covariance and \eqref{eq:rank-one-nica-product} give
\begin{equation}
\label{eq:range-projection-intersection}
 S_aS_a^*S_bS_b^*
   =\sum_{\gamma\in\MCE_{\G}(a,b)}S_\gamma S_\gamma^*.
\end{equation}
The summands are mutually orthogonal: two distinct representatives have
zero inner product by \eqref{eq:intrinsic-inner-product}.  Thus the sum in
\eqref{eq:range-projection-intersection} is the join of the range
projections.  These are precisely the relations in
\cite[Theorem~9.7]{Spielberg2020}, so there is a homomorphism
$C^*(\mathscr G_2(\Ccat))\to\NT_{X(\Ccat)}$ taking $T_c$ to $S_c$.

Conversely, let $\{T_c:c\in\Ccat\}$ be Spielberg's universal family.  For
$g\in\G$, one has $g\Ccat=r(g)\Ccat$, so the range-projection relation gives
$T_gT_g^*=T_{r(g)}$.  The source and inverse relations then give
\[
 T_{g^{-1}}=(T_g^*T_g)T_{g^{-1}}
 =T_g^*(T_gT_{g^{-1}})=T_g^*T_{r(g)}=T_g^*.
\]
Thus the $T_g$ extend to a nondegenerate representation of
$A=C^*(\G)$.  Set
\[
                         \psi_{\bm n}(\xi_c)=T_c.
\]
The module and multiplication identities follow from category
multiplication.  If $c,d$ have the same degree and $d=cg$, then
\[
                         T_c^*T_d=T_g.
\]
If $c\G\ne d\G$, then $c\Ccat\cap d\Ccat$ is empty: a common extension
would, by weak factorization at their common degree, give $d=cg$ for an
invertible $g$.  Spielberg's range-projection relation therefore gives
$T_cT_c^*T_dT_d^*=0$, and hence
\[
 T_c^*T_d=T_c^*(T_cT_c^*T_dT_d^*)T_d=0.
\]
Hence $\psi$ preserves the inner products and is a Toeplitz
representation of $X(\Ccat)$.

For Nica covariance it suffices to use rank-one operators as in
Proposition~\ref{prop:finite-alignment-compacts}.  Spielberg's
range-projection relation and
$\gamma=b\beta_\gamma=c\delta_\gamma$ give
\begin{align*}
 T_aT_b^*T_cT_d^*
 &=T_aT_b^*(T_bT_b^*T_cT_c^*)T_cT_d^*\\
 &=\sum_{\gamma\in\MCE_{\G}(b,c)}
       T_{a\beta_\gamma}T_{d\delta_\gamma}^*.
\end{align*}
By \eqref{eq:rank-one-nica-product}, this is the Nica relation.  Thus
$\psi$ induces a homomorphism in the reverse direction.  The two maps fix
every arrow generator and are inverse.
\end{proof}

\subsection{The boundary quotient}

Call $\Ccat$ \emph{row-finite modulo $\G$} if
$v\Ccat^{\bm n}/\G$ is finite for every $v\in\Ccat^0$ and
$\bm n\in\Nat^k$.  It has \emph{no sources} if
$v\Ccat^{\bm n}$ is nonempty for every such $v$ and $\bm n$.  Choose a
finite right-orbit transversal
\[
                F(v,\bm n)\subseteq v\Ccat^{\bm n}.
\]
The proof of Lemma~\ref{lem:mce-degree}, applied at the common range of two
arrows, shows that row-finiteness modulo $\G$ implies finite alignment.

\begin{lemma}
\label{lem:compact-left-action-formula}
If $\Ccat$ is row-finite modulo $\G$, then every left action is by compact
operators.  For $g\in\G$ and nonzero $\bm n$,
\begin{equation}
\label{eq:compact-left-action-formula}
 \varphi_{\bm n}(u_g)
   =\sum_{x\in F(s(g),\bm n)}
       \Theta_{\xi_{gx},\xi_x}.
\end{equation}
\end{lemma}

\begin{proof}
On the right orbit of $x$, the corresponding summand maps
$\xi_{xh}$ to $\xi_{gxh}$, and it vanishes on every other orbit.  The sum is
finite and therefore equals the left action of $u_g$.  The groupoid
generators span a dense subalgebra of $A$, and the compact operators are
closed.  In degree zero, the assertion is the standard identity
$\compacts(A_A)=A$.
\end{proof}

Let $\mathcal O^{\mathrm F}_{X(\Ccat)}$ denote Fowler's Cuntz--Pimsner
algebra of the product system.  Under the hypothesis of Lemma
\ref{lem:compact-left-action-formula}, its covariance relation is
\begin{equation}
\label{eq:fowler-covariance}
 \psi^{(\bm n)}(\varphi_{\bm n}(a))=\psi_{\bm 0}(a),
 \qquad a\in A,\quad \bm n\ne\bm 0,
\end{equation}
as in \cite[Definition~2.5 and Proposition~2.9]{Fowler2002}.

For a countable finitely aligned category, Spielberg defines the boundary
algebra
\[
 \mathcal O_{\mathrm{Sp}}(\Ccat)
   =C^*(\mathscr G_2(\Ccat)|_{\partial\Ccat}).
\]
It is the universal quotient of $\Tcal(\Ccat)$ satisfying
\begin{equation}
\label{eq:spielberg-exhaustive-relation}
 T_v=\bigvee_{\alpha\in E}T_\alpha T_\alpha^*
\end{equation}
for every finite exhaustive set $E\subseteq v\Ccat$
\cite[Definitions~10.13--10.14 and Theorem~10.15]{Spielberg2020}.  Here
$E$ is exhaustive when each $c\in v\Ccat$ has a common extension with some
$\alpha\in E$.

\begin{theorem}
\label{thm:fowler-spielberg-boundary}
Suppose that $\Ccat$ is row-finite modulo $\G$ and has no sources.  Then there
is a canonical isomorphism
\begin{equation}
\label{eq:fowler-spielberg-boundary}
 \mathcal O^{\mathrm F}_{X(\Ccat)}
     \cong C^*(\mathscr G_2(\Ccat)|_{\partial\Ccat})
\end{equation}
that identifies the images of every $\xi_c$ and $T_c$.
\end{theorem}

\begin{proof}
Let $\psi$ be a Toeplitz representation and use the notation $U_g,S_c$ from
Proposition~\ref{prop:representation-relations}; put
$P_v=U_v$ and $P_c=S_cS_c^*$.  By
\eqref{eq:compact-left-action-formula}, Fowler covariance implies
\begin{equation}
\label{eq:fixed-degree-cuntz-relation}
                   P_v=\sum_{x\in F(v,\bm n)}P_x
\end{equation}
for every vertex $v$ and nonzero degree $\bm n$.  Conversely, these vertex
relations imply \eqref{eq:fowler-covariance} on every groupoid generator:
\begin{align*}
 U_g
 &=U_gP_{s(g)}
   =\sum_{x\in F(s(g),\bm n)}S_{gx}S_x^*\\
 &=\psi^{(\bm n)}(\varphi_{\bm n}(u_g)).
\end{align*}
They therefore imply Fowler covariance on all of $A$.

The fixed-degree relations are equivalent to Spielberg's finite
exhaustive-set relations.  The set $F(v,\bm n)$ is finite and exhaustive.
Given
$c\in v\Ccat$, put
$\bm q=(\ell(c)\vee\bm n)-\ell(c)$, choose
$d\in s(c)\Ccat^{\bm q}$, and refactor
$cd=ab$ with $\ell(a)=\bm n$.  Writing $a=xg$ with
$x\in F(v,\bm n)$ gives $cd=x(gb)$, so $c$ and $x$ have a common
extension.  Distinct elements of $F(v,\bm n)$ have orthogonal range
projections, so \eqref{eq:fixed-degree-cuntz-relation} is the boundary
relation for this exhaustive set.

Conversely, let $E\subseteq v\Ccat$ be finite and exhaustive, and choose
$\bm n\ne\bm 0$ with $\bm n\geq\ell(\alpha)$ for every $\alpha\in E$.
For $x\in F(v,\bm n)$, choose $\alpha\in E$ having a common extension with
$x$.  Weak factorization at degree $\bm n$ shows that $x\in\alpha\Ccat$:
if $xy=\alpha z$, factor $z=z_1z_2$ with
$\ell(z_1)=\bm n-\ell(\alpha)$ and compare the two degree-$\bm n$
prefixes to obtain $\alpha z_1=xg$ for some $g\in\G$.  Hence
$P_x\leq P_\alpha$.  Relation
\eqref{eq:fixed-degree-cuntz-relation} now gives
\[
 P_v=\sum_{x\in F(v,\bm n)}P_x
     \leq\bigvee_{\alpha\in E}P_\alpha\leq P_v,
\]
which is \eqref{eq:spielberg-exhaustive-relation}.  Thus the fixed-degree
relations and all the finite exhaustive-set relations are equivalent.

Every fiber is essential by Proposition~\ref{prop:left-action}, the semigroup
$\Nat^k$ is directed, and all the left actions are compact.  Fowler's
\cite[Proposition~5.4]{Fowler2002} therefore shows that every
Fowler-covariant representation is automatically Nica covariant.  Theorem
\ref{thm:nica-spielberg}, Spielberg's
\cite[Theorem~10.15]{Spielberg2020}, and the equivalence of covariance
relations just proved give \eqref{eq:fowler-spielberg-boundary}.
\end{proof}

The fixed-degree relation is independent of the transversal: replacing $x$
by $xg$ does not change its range projection, since
\[
             S_{xg}S_{xg}^*=S_xU_gU_g^*S_x^*=S_xS_x^*.
\]

For the Cuntz--Nica--Pimsner quotient, assume in addition that every left
action is injective.

\begin{corollary}
\label{cor:cnp-spielberg-boundary}
Under the hypotheses of Theorem~\ref{thm:fowler-spielberg-boundary}, suppose
also that every left action $\varphi_{\bm n}$ is injective.  Then
\[
 \NOcal_{X(\Ccat)}
   \cong\mathcal O^{\mathrm F}_{X(\Ccat)}
   \cong C^*(\mathscr G_2(\Ccat)|_{\partial\Ccat}).
\]
\end{corollary}

\begin{proof}
The product system is compactly aligned by Proposition
\ref{prop:finite-alignment-compacts}, and its left actions are compact by
Lemma~\ref{lem:compact-left-action-formula}.  With injectivity,
\cite[Proposition~5.1 and Corollary~5.2]{SimsYeend2010} identifies
Cuntz--Nica--Pimsner covariance with Fowler covariance.  Apply
Theorem~\ref{thm:fowler-spielberg-boundary}.
\end{proof}

The left actions in Proposition~\ref{prop:left-action} need not be injective.
Without injectivity, the cited Sims--Yeend comparison does not apply, and the
Cuntz--Nica--Pimsner quotient requires a separate analysis.
 \section{Compatible coefficient completions}
\label{sec:further-questions}

The construction above uses the full groupoid algebra.  Replacing it by a
quotient requires compatibility with the left actions and multiplication
maps.

Retain the assumptions that $(\Ccat,\ell)$ is countable and left
cancellative, put $\G=\Ccat^\times$, and let $X(\Ccat)$ be the product system
of Theorem~\ref{thm:linearization}.

Let $q\colon C^*(\G)\to B$ be a quotient.  For each degree, the right Hilbert
$B$-module obtained from $X_{\bm n}$ is
\[
                         X_{\bm n}^B=X_{\bm n}\otimes_q B.
\]
Write $X_{\bm n}\ker q$ for the closed span of
$\{xa:x\in X_{\bm n},\ a\in\ker q\}$.  The left action descends through
$q$ precisely when
\begin{equation}
\label{eq:descent-condition}
 \varphi_{\bm n}(\ker q)X_{\bm n}
       \subseteq X_{\bm n}\ker q.
\end{equation}

\begin{definition}
\label{def:compatible-completion}
The quotient $B$ is \emph{$\Ccat$-compatible} if
\eqref{eq:descent-condition} holds in every degree and the algebraic maps
induced by category multiplication extend to correspondence unitaries
\[
 X_{\bm m}^B\otimes_BX_{\bm n}^B
       \longrightarrow X_{\bm m+\bm n}^B
\]
forming a product system.
\end{definition}

The identity quotient is compatible by Theorem~\ref{thm:linearization}.  A
general quotient, including the reduced groupoid algebra, is not compatible
without a descent argument.  In the one-object, rank-one, finite-alphabet
setting, Nekrashevych defines self-similar completions and proves that the
generic-orbit completion $A_\Phi$ is a quotient of every other such completion
\cite[Definition~3.4 and Theorem~5.2]{Nekrashevych2004}.  His proof uses the
specific boundary dynamics of a countable, level-transitive self-similar group
action.  It does not by itself produce a minimal $\Ccat$-compatible quotient
in several degrees.

\begin{question}
\label{ques:minimal-completion}
For which generalized higher-rank graphs does $C^*(\G)$ have a smallest
$\Ccat$-compatible quotient?  Can it be described by representations on a
boundary path space, and when is the reduced groupoid algebra compatible?
\end{question}
 \section{Lean formalization}
\label{sec:lean-formalization}

An accompanying Lean~4 development verifies the categorical and combinatorial
part of the paper.  It compiles with Lean~4.28.0 and Mathlib~v4.28.0.  A
transitive audit of the imported declarations finds no unfinished proofs or
additional axioms beyond \texttt{propext}, \texttt{Classical.choice}, and
\texttt{Quot.sound}.

The formalization includes Definition~\ref{def:generalized-graph},
Lemma~\ref{lem:right-action-free}, Definition~\ref{def:biset-product-system},
Proposition~\ref{prop:multiplication-balanced}, and both constructions in
Theorem~\ref{thm:wfp-biset-system}.  It also verifies the categorical
assertions in Example~\ref{ex:heisenberg}: associativity, weak factorization,
left and right cancellation, the invertible subgroupoid, and failure of the
$R$-condition for every atomic transversal.

For Theorem~\ref{thm:r-condition-splitting}, Lean proves that atoms are the
arrows of standard-generator degree, that every arrow has a normal form
$\lambda g$, that the $R$-condition is equivalent to strict splitting, and
that a strict splitting gives unique factorization inside $\Lambda$.  The
normal-form multiplication identity is also proved.  The development does
not define the action and restriction maps, verify the Zappa--Sz\'ep axioms,
or construct the size-preserving isomorphism
$\Ccat\cong\Lambda\bowtie\G$ asserted in the last sentence of the theorem.

Several formalized lemmas supply parts of later analytic arguments.  These
are the uniqueness of the transversal decomposition $c=tg$ used in
Proposition~\ref{prop:coordinate-module}; the minimal-common-extension degree
criterion in Lemma~\ref{lem:mce-degree}; the fact that equal-degree arrows
with a common extension lie in the same right $\G$-orbit; and the
common-extension and divisibility statements used in the proof of
Theorem~\ref{thm:fowler-spielberg-boundary}.  The implication from
orbit-finiteness to finite alignment in Lemma~\ref{lem:mce-degree} is not
formalized.

The current development proves only one round trip associated with
Corollary~\ref{cor:categorical-isomorphism}: a generalized higher-rank graph
is equivalent to the category reconstructed from its homogeneous fibers.  It
does not define morphisms between generalized higher-rank graphs or between
biset product systems, and it does not prove the reverse round trip starting
from a biset product system.  The categorical isomorphism asserted in the
corollary is therefore not yet formalized.

The operator-algebraic layer is outside the present Lean development: Hilbert
$C^*$-modules, correspondences, groupoid $C^*$-algebras, Nica--Toeplitz and
Cuntz--Pimsner algebras, and Spielberg's groupoids are not defined.  Hence
Proposition~\ref{prop:coordinate-module} is checked only through its
combinatorial decomposition.  Propositions~\ref{prop:left-action},
\ref{prop:degree-one-coherence}, \ref{prop:representation-relations},
\ref{prop:heisenberg-product-system}, and
\ref{prop:finite-alignment-compacts}; Theorems~\ref{thm:linearization},
\ref{thm:strict-correspondence-comparison}, \ref{thm:nica-spielberg}, and
\ref{thm:fowler-spielberg-boundary}; Corollaries~\ref{cor:compact-alignment}
and \ref{cor:cnp-spielberg-boundary}; and
Lemma~\ref{lem:compact-left-action-formula} have not been formalized as
stated.  Section~\ref{sec:further-questions} is also outside the development.

\section*{Acknowledgments}
This work was initiated at the Banff International Research Station (BIRS)
during the workshop
\href{https://www.birs.ca/events/2023/5-day-workshops/23w5033}
{\emph{Joint Spectra and related Topics in Complex Dynamics and Representation
Theory}}, May 21--26, 2023.  The authors thank BIRS and the workshop organizers
for their hospitality.

\providecommand{\bysame}{\leavevmode\hbox to3em{\hrulefill}\thinspace}
\providecommand{\MR}{\relax\ifhmode\unskip\space\fi MR }
\providecommand{\MRhref}[2]{\href{http://www.ams.org/mathscinet-getitem?mr=#1}{#2}
}
\providecommand{\href}[2]{#2}

\end{document}